\documentclass[12pt]{amsart}
\usepackage[square,numbers]{natbib}  %Nature-like bibliography
\usepackage{graphicx}
\usepackage{verbatim}
\usepackage{amsmath}
\usepackage{amssymb, mathrsfs}
\usepackage{amsbsy}
\usepackage{amscd}
\usepackage{amsthm}
\usepackage{bm}

\theoremstyle{plain}
\newtheorem{thm}{Theorem}[section]                                          
\newtheorem{prop}[thm]{Proposition}                          

\newtheorem{cor}[thm]{Corollary}
\theoremstyle{definition}

\theoremstyle{remark}
\newtheorem{rem}[thm]{Remark}

\makeatletter \@addtoreset{equation}{section} \makeatother

\newcommand{\N}{\mathbb{N}}

\newcommand{\Z}{\mathbb{Z}}

\newcommand{\set}[1]{\left\{#1\right\}}

\newcommand{\cP}{\mathcal{P}}

\newcommand{\cF}{\mathscr{F}}
\newcommand{\Min}{{\rm{Min}}}

\newcommand{\conn}{\leadsto}

\begin{document}

\title[Exponential growth for linear systems]
{On exponential growth for a certain class of linear systems}
\author{Ryoki Fukushima}
\thanks{The first author was supported in part by JSPS Fellowships 
for Young Scientists and JSPS Grant-in-Aid for Research Activity 
Start-up 22840019}
\address{
{\rm Ryoki Fukushima}\\
Department of Mathematics, Graduate School of Science and Engineering, 
Tokyo Institute of Technology, Tokyo 152-8551, Japan.}
\email{\tt ryoki@math.titech.ac.jp}

\author{Nobuo Yoshida}
\thanks{The second author was supported in part by JSPS Grant-in-Aid 
for Scientific Research, Kiban (C) 21540125}
\address{
{\rm Nobuo Yoshida}\\
Division of Mathematics, Graduate School of Science, Kyoto University, 
Kyoto 606-8502, Japan.}
\email{\tt nobuo@math.kyoto-u.ac.jp}

\begin{abstract}
We consider a class of stochastic growth models on the integer lattice 
which includes various interesting examples such as the number of open paths
in oriented percolation and the binary contact path process.
Under some mild assumptions, we show that the total mass of the process 
grows exponentially in time whenever it survives. 
More precisely, we prove that there exists an open path, oriented in time,
along which the mass grows exponentially fast.
\end{abstract}

\maketitle

\section{Introduction}
\label{Intro}
\subsection{Overview}
We consider a class of stochastic growth models on the integer lattice
$\Z^d$ which includes a time discretization and a special 
case of the `linear systems' 
discussed in Chapter IX of Liggett's book~\cite{Lig85}. 
One of the simplest examples is the number of distinct open paths on 
the cluster of contact process studied  
by~\cite{Gri83} and as is discussed there, it can be thought of
as a model of population growth with spatial structure. 
There has recently been some progress on this type of models such as; 
phase transition for the growth rate of total population~\cite{Yos08};
diffusivity~\cite{Nak10,NY09,NY10a} or localization~\cite{Yos10,NY10b}
of the population density.
However, the following fundamental question remains:
does the total population grow exponentially whenever it survives?
It is well known that the answer is affirmative for
the classical Galton-Watson process 
(see, e.g., Corollary~1.6 on p.20 in~\cite{AH83}). 
In this paper, we show that the same assertion holds for a fairly 
general growth models with spatial structure.
In fact, we show that there exists a single path along 
which the population grows exponentially. 
\begin{rem}
We say that a process survives if there always exists at least one
particle, which seems natural.
Note, however, that in the theory of linear systems, 
the term survival 
often refers to the stronger condition that the total population 
grows as fast as its expectation, 
see e.g.~Theorem~2.4 on p.433 in~\cite{Lig85}. 
\end{rem}

\subsection{Setting and main results}
\label{Setting and main results}
Let us start by describing the definition of the process. 
Although we have results for both discrete and continuous time
processes, we first focus on the discrete time case and 
discuss the continuous time case in Section~\ref{application}. 
We write $\N$ for the set of nonnegative integers and $\N^*$ for 
$\N \setminus \{0\}$. 
Let $B=(B_{x,y})_{x,y \in \Z^d}$ be a random matrix of infinite 
size whose entries take values in $\{0\}\cup [1,\infty)$. 
We assume that $B$ is translation invariant in the sense that 
$(B_{x,y})_{x,y \in \Z^d}$ and 
$(B_{x+z,y+z})_{x,y \in \Z^d}$ has the same law for any $z \in \Z^d$. 
Using independent copies $\{B_n\}_{n \in \N^*}$
of $B$, we define a Markov chain 
$\{M_n\}_{n \in \N}=\{(M_{n,x})_{x \in \Z^d}\}_{n \in \N}$ as follows: 
\begin{equation}
	M_{0,x}=\delta_{o, x} \textrm{ and }
	M_{n, x}=\sum_{y \in \Z^d} M_{n-1, y}B_{n, y, x} \textrm{ for }
        n \in \N^*, \label{evol}
\end{equation}
where $o$ denotes the origin of $\Z^d$ and $\delta_{x,y}$ the Kronecker delta: 
\begin{equation}
	\delta_{x,y}=
	\begin{cases}
		1 & \textrm{if } x=y, \\
		0 & \textrm{if } x \neq y.
	\end{cases}
\end{equation}
The resulting process is $[0,\infty]^{\Z^d}$-valued since the sum
in~\eqref{evol} may diverge. 
If we regard $M_n$ as a row vector, we can rewrite the above equation as
\begin{equation}
	M_n=(\delta_{o, x})_{x \in \Z^d} B_1 \cdots B_n.
\end{equation}
for $n \in \N^*$.   
We denote the total mass of the process by 
\begin{equation}
	|M_n|\stackrel{\rm def} = \sum_{x \in \Z^d}M_{n, x}.
\end{equation}
We call a sequence $\{\Gamma(n)\}_{n=k}^{l}
\subset \Z^d$ ($k < l \le \infty$) 
an open path if $B_{n+1, \Gamma(n), \Gamma(n+1)}\ge 1$ 
for all $k \le n < l$. 
For the sake of shorthand, let
\begin{equation}
	c_{\delta}(B)= P(|M_n| \ge 1 \textrm{ \upshape for all }
        n \in \N)\sup_{x \in \Z^d}P(B_{o,x} \ge 1+\delta) \label{rate}
\end{equation}
for $\delta>0$. Now we are in position to state our main result.
%%% Main Theorem %%%
\begin{thm}\label{main}
Suppose that there exists $\delta>0$ such that $c_{\delta}(B_1)>0$. 
Then for each $\epsilon>0$, 
there exists a random open path $\{\Gamma(n)\}_{n \in \N}$ such that
\begin{equation}
	\liminf_{n \to \infty}\frac{1}{n} \log M_{n, \Gamma(n)} 
	\ge c_{\delta-\epsilon}(B_1) 
	\log(1+\delta-\epsilon)\label{growth} 
\end{equation}
almost surely on $\{|M_n| \ge 1 \textrm{ \upshape for all }n \in \N\}$. 
In particular, we have 
\begin{equation}
	\liminf_{n \to \infty}\frac{1}{n} \log |M_n| 
	\ge c_{\delta}(B_1) \log(1+\delta)\label{growth2}
\end{equation}
almost surely on $\{|M_n| \ge 1 \textrm{ \upshape for all }n \in \N\}$
(since $c_{\delta-\epsilon}(B_1)\ge c_{\delta(B_1)}$ and by the
continuity of the logarithm).
\end{thm}
%%%%%%%%%%%%%%
The following corollary is more useful in applications than 
Theorem~\ref{main} 
(see Section~\ref{application} below). 
In what follows, we denote the matrix product 
$B_1 \cdots B_m$ by $\prod_{k=1}^m B_k$. 
%%%%%%%%%%%%%%
\begin{cor}\label{iteration}
Suppose that there exist $\delta>0$ and $m \in \N^*$ such that 
$c_{\delta}(\prod_{k=1}^m B_k)>0$. 
Then for each $\epsilon>0$,
there exists a random open path $\{\Gamma(n)\}_{n \in \N}$ such that
\begin{equation}
	\liminf_{n \to \infty}\frac{1}{n} \log M_{n, \Gamma(n)} 
	\ge \frac{1}{m}c_{\delta-\epsilon}({\textstyle \prod_{k=1}^m} B_k) 
	\log(1+\delta-\epsilon)\label{growth}
\end{equation}
almost surely on $\{|M_n| \ge 1 \textrm{ \upshape for all }n \in \N\}$. 
In particular, we have
\begin{equation}
	\liminf_{n \to \infty}\frac{1}{n} \log |M_n| 
	\ge \frac{1}{m}c_{\delta}({\textstyle \prod_{k=1}^m} B_k) 
        \log(1+\delta)
\end{equation}
almost surely on $\{|M_n| \ge 1 \textrm{ \upshape for all }n \in \N\}$.
{Suppose on the other hand 
that $c_{\delta}(\prod_{k=1}^m B_k)=0$ for all $\delta>0$ and
$m \in \N$ and that there exists $r_B>0$ such that 
$B_{x,y}=0$ if $|x-y| \ge r_B$. Then} 
\begin{equation}
	{\limsup_{n \to \infty}}\frac{1}{n} \log |M_n| 
	\le 0\label{nongrowth}
\end{equation}
almost surely with the convention $\log 0=-\infty$.
\end{cor}
\begin{rem}
The main point of the above results is its generality. 
We put no independence assumptions on the 
elements of $(B_{x,y})_{x,y\in\Z^d}$, 
nor the finite range assumption 
except for the last assertion in Corollary~\ref{iteration}.
This for instance allows us to obtain analogous 
results for certain continuous time models by simply
applying the discrete time results
(see Subsection~\ref{CT-analogue}). 
\end{rem}
\begin{rem}
As a special case of our results, it follows that the number of 
open paths of length $n$ in 
supercritical oriented percolation grows exponentially in $n$
(see Subsection~\ref{LSE} below).
In this special case, the following interesting result has recently
obtained by~\cite{KNPS10}, which is valid also in the subcritical
phase. Consider the oriented paths of length $n$ which 
go through a maximal number of open sites. 
Then, the number of such {\it maximal} paths grows exponentially in
$n$ for all $p>0$. However, 
this work seems to have only limited overlap with ours 
since they mainly focus on the subcritical phase where 
the number of such paths does not 
obey the evolution rule~\eqref{evol}. 
\end{rem}

%%% Construction %%%
\section{Proof of Theorem~\ref{main} and Corollary~\ref{iteration}}
\label{proof}
We prove Theorem~\ref{main} and Corollary~\ref{iteration} 
in this section.
Let us briefly explain the strategy to prove Theorem~\ref{main}. 
We are going to find an infinite open 
path $\Gamma$ which goes through 
{\it heavy bonds}, 
i.e.~$B_{n+1,\Gamma(n),\Gamma(n+1)}\ge 1+\delta{-\epsilon}$, 
many times.
To this end, we first construct a path $\gamma$ which is 
not necessarily open but it prefers to go through heavy bonds and 
its construction uses only local information, that is, 
it does not refer to the future and also not too much to the past.
Next, we consider certain good events
$\{G_n(\gamma)\}_{n \in \N}$ such that
\begin{enumerate}
 \item if $G_n$ happens, then 
       $B_{n+1,\gamma(n),\gamma(n+1)}\ge 1+\delta{-\epsilon}$ and
 \item there exists an infinite open path which shares all
       the bonds where $G_n$ happens with $\gamma$.
\end{enumerate}
Then we can prove, essentially due to the locality of $\gamma$, 
the law of large numbers for $\{1_{G_n(\gamma)}\}_{n \in \N}$
and this ensures the existence of the above $\Gamma$.

The proof of Corollary~\ref{iteration} will be given 
in the final subsection.

We introduce the notation used in the sequel. For $x, y \in \Z^d$ 
and $m, n \in \N$, 
we write 
\begin{equation}
	(m, x) \conn (m+n, y)
\end{equation}
if there exists an open path $\{x(k)\}_{k=m}^{m+n}$ with $x(m)=x$ 
and $x(m+n)=y$. 
We adopt the convention that $(m, x) \conn (m, x)$. 
We define the process started from $(m,x) \in \N \times \Z^d$ by 
\begin{equation}
	M_{m,y}^{(m,x)}=\delta_{x, y} \textrm{ and }
	M_{n, y}^{(m,x)}=\sum_{z \in \Z^d} M_{n-1, z}B_{n, z, y} 
        \textrm{ for }n >m.\label{restart}
\end{equation}
We call $(m,x) \in \N \times \Z^d$ a {\it{percolation point}} if 
\begin{equation}
	|M_n^{(m,x)}| \ge 1 \textrm{ for all } n > m
\end{equation}
and write $\cP$ for the set of all percolation points. 
Finally, we introduce the sigma-fields
\begin{equation}
	\cF_{m,n}=\sigma\left[B_{k, x,y}: 
        m \le k \le n, x, y \in \Z^d \right]
\end{equation}
for $m,n \in \N^*$ with $m \le n$.

\subsection{Construction of the path}\label{construction}
We assume for simplicity that 
\begin{equation}
 \textrm{there exists a site }x \in \Z^d 
 \textrm{ which maximizes }P(B_{o,x} \ge 1+\delta).\label{ass2.1} 
\end{equation}
Otherwise, pick a site for which 
$P(B_{o,x} \ge 1+\delta-\epsilon){>0}$
and replace $\delta$ by
$\delta-\epsilon$ in what follows.
{Let us fix an enumeration of $\Z^d$ and 
write $\Min A$ for the first 
element appearing in $A \subset \Z^d$.}
We define a path 
$\gamma=\{\gamma(n)\}_{n \in \N}$ 
according to the following recursive algorithm: 
\begin{enumerate}
\renewcommand{\labelenumi}{(\roman{enumi})}
\item{Let $\gamma(0)=o$.}
\item{If $(n, \gamma(n)) \conn (n+1, \gamma(n)+x)$, 
then let $\gamma(n+1)=\gamma(n)+x$.}
\item{If $(n, \gamma(n)) \not\conn (n+1, \gamma(n)+x)$ 
and $|M_{n+1}^{(n, \gamma(n))}| \ge 1$, 
then let $$\gamma(n+1)=\gamma(n)+\Min\{y \in \Z^d: 
M_{n+1, \gamma(n)+y}^{(n, \gamma(n))}\ge 1\}.$$}
\item{If $|M_{n+1}^{(n, \gamma(n))}| =0$ and 
$\{k \in \N: k \le n,  |M_{n+1}^{(k, \gamma(k))}| \ge 1\} 
\neq \emptyset$, then let 
$$
	T_n=\max\{k \in \N: k \le n, 
        |M_{n+1}^{(k, \gamma(k))}| \ge 1\}
$$
and 
$$
	\gamma(n+1)=\gamma(T_n)+\Min\{y \in \Z^d: 
        M_{n+1, {\gamma(T_n)+}y}^{(T_n, 
        \gamma(T_n))}\ge 1\}
$$}
\item{If $|M_{n+1}^{(n, \gamma(n))}| =0$ and 
$\{k \in \N: k \le n,  |M_{n+1}^{(k, \gamma(k))}| \ge 1\}
=\emptyset$, then let 
$\gamma(n+1)=o.$}
\end{enumerate}
Also for each $(m,v)\in \N\times \Z^d$, we define a path 
$\gamma^{(m, v)}=\{\gamma^{(m, v)}(n)\}_{n \ge m}$ 
in the same way as above but we let 
$\gamma^{(m, v)}(m)=v$ in (i), 
restrict the ranges of $k$ to $m \le k \le n$ in (iv) and (v),
and let $\gamma^{(m, v)}(n+1)=v$  in (v).
We denote by $T_n^{(m,v)}$ the corresponding $T_n$.

The following properties are obvious from the construction: 
\begin{align}
	&(n, \gamma^{(m, v)}(n)) \conn (n+1, \gamma^{(m, v)}(n)+x) 
        \textrm{ implies }
	\gamma^{(m, v)}(n+1)=\gamma^{(m, v)}(n)+x, \label{preference} \\
	&(\gamma^{(m, v)}(m), \ldots, \gamma^{(m, v)}(n))\textrm{ is }
        \cF_{m+1,n}\textrm{-measurable}.\label{measurability} 
\end{align}
We also know that the construction does not go back beyond a 
percolation point, 
which will be crucial in the proof of Theorem~\ref{main}: 
%%% Prop 1 %%%
\begin{prop}\label{renewal}
Let $\gamma=\gamma^{(0,o)}$ be the path constructed above 
{and
$l, m \in \N$. 
Then on the event $|M_{m+l}^{(m, \gamma(m))}| \ge 1$,}
\begin{enumerate}
\renewcommand{\labelenumi}{(\alph{enumi})}
\item{$\gamma(n)=\gamma^{(m, \gamma(m))}(n)$ for all 
$n = m, m+1, \ldots , m+l$, }
\item{$(m, \gamma(m)) \conn (n, \gamma(n))$ for all 
$n = m, m+1, \ldots , m+l$.}
\end{enumerate}
In particular, if $(m, \gamma(m)) \in \cP$, then {\upshape (a)} 
and {\upshape (b)} hold for all $n \ge m$. 
\end{prop}
%%%%%%%%%%
\begin{proof}
We prove (a) and (b) simultaneously by induction on 
$n =m, m+1, \ldots, m+l$. 
They are obviously true for $n=m$. 
Suppose that the claims hold up to some $n \ge m$. 
If  $|M_{n+1}^{(n, \gamma(n))}| \ge 1$, then both 
$\gamma(n+1)$ and $\gamma^{(m, \gamma(m))}(n+1)$ 
are chosen by (ii) or (iii) in the algorithm. Since 
$\gamma(n)=\gamma^{(m, \gamma(m))}(n)$ by the induction 
hypothesis, they are chosen in the same manner and thus 
$\gamma(n+1)=\gamma^{(m, \gamma(m))}(n+1)$. 
Moreover, we have $(n, \gamma(n)) \conn (n+1,\gamma(n+1))$ 
in this case and hence it follows that
\begin{equation}
	(m, \gamma(m)) \conn (n, \gamma(n)) \conn (n+1,\gamma(n+1))
\end{equation}
by the induction hypothesis. 
If, on the other hand, $|M_{n+1}^{(n, \gamma(n))}| =0$, then note that 
$|M_{n+1}^{(m, \gamma(m))}| \ge 1$ as long as $n+1 \le m+l$ 
by the assumption. 
In particular, it follows that $\gamma(n+1)$ is chosen by (iv) 
in the algorithm and $m \le T_n \le n$. 
Then the induction hypothesis shows that (I) $T_n=T^{(m,\gamma(m))}_n$, 
(II) $\gamma(T_n)=\gamma^{(m, \gamma(m))}(T^{(m,\gamma(m))}_n)$, and 
(III) $(m, \gamma(m)) \conn (T_n,\gamma(T_n)) $. 
Since we have
\begin{equation}
\begin{split}
	\gamma(n+1)&=\Min\{y \in \Z^d: 
        M_{n+1, y{+\gamma(T_n)}}^{(T_n, \gamma(T_n))}\ge 1\}\\
	&=\Min\{y \in \Z^d: 
        M_{n+1, y{+\gamma^{(m, \gamma(m))}(T^{(m,\gamma(m))}_n))}}
        ^{(T^{(m,\gamma(m))}_n, 
        \gamma^{(m, \gamma(m))}(T^{(m,\gamma(m))}_n))}\ge 1\}\\
	&=\gamma^{(m, \gamma(m))}(n+1)
\end{split}
\end{equation}
by (I) and (II) and
\begin{equation}
	(m, \gamma(m)) \conn (T_n,\gamma(T_n)) \conn (n+1, \gamma(n+1))
\end{equation}
by (III), the proof is complete.
\end{proof}

Finally, we construct an open path $\Gamma$ on $\{(0,o) \in \cP\}$ 
by connecting percolation points
on $\gamma$. Assume $(0,o) \in \cP$ and let $\tau_1=0$ and 
\begin{equation}
	\tau_{n+1}=\inf\{k > \tau_n: (k, \gamma(k)) \in \cP\}
\end{equation}
for $n \ge 1$, so that $(\tau_n, \gamma(\tau_n))$ are nothing but 
the $n$-th percolation points on $\gamma$. 
We may assume that $\tau_n<\infty$ for all $n \in \N^*$ since it will 
be proved in Proposition~\ref{LLN} below. 
(Note that Proposition~\ref{LLN} refers to $\gamma$ only.)
Let us first set $\Gamma(\tau_n)=\gamma(\tau_n)$ for all $n \in \N^*$. 
Next for $\tau_n < k < \tau_{n+1}$ (if any), we define $\Gamma(k)$ 
recursively as follows: Given $\Gamma(k-1)$, let 
\begin{equation}
	\Gamma(k)=\Min\{y \in \Z^d: (k-1, \Gamma(k-1)) \conn 
        (k, y) \conn (\tau_{n+1}, \gamma(\tau_{n+1}))\}.\label{Gamma-const} 
\end{equation}
{Note that 
$(\tau_n, \gamma(\tau_n)) \conn (\tau_{n+1}, \gamma(\tau_{n+1}))$
by Proposition~\ref{renewal}-(b).
Then, by induction on $k$, we see that the set on the 
right-hand side of~\eqref{Gamma-const} is nonempty and 
that $(k-1, \Gamma(k-1)) \conn (k, \Gamma(k))$ for all 
$k\in (\tau_n, \tau_{n+1}]$.}
\begin{rem}
A similar construction was used by~\cite{Kuc89} to show 
the central limit theorem for 
the right edge of (1+1)-dimensional supercritical oriented percolation. 
\end{rem}

%%% Mixing %%%
\subsection{Law of large numbers for good events}\label{LLN for G_n}
Throughout this subsection, 
we keep assuming~\eqref{ass2.1}
and write $\gamma$ and $\Gamma$ for the 
paths constructed in Subsection~\ref{construction}. 
We consider the following {\it good} event 
\begin{equation}
	G_n(\gamma)=\{ B_{n+1, \gamma(n), \gamma(n)+x} 
        \ge 1+\delta, (n+1, \gamma(n+1)) \in \cP\}
\end{equation}
for each $n \in \N$. Note that on $G_n(\gamma)$, we have 
$\gamma(n+1)=\gamma(n)+x$ by~\eqref{preference} 
and $(n, \gamma(n)) \in \cP$. 
It follows in particular that 
\begin{equation}
	\gamma(n)=\Gamma(n)\textrm{ and }\gamma(n+1)=\Gamma(n+1)
\end{equation}
on $G_n(\gamma)\cap\{(0,o) \in \cP\}$. 
The following proposition shows that $M_{n, \Gamma(n)}$ is multiplied 
by at least $(1+\delta)$ when $G_n(\gamma)$ occurs 
and thus explains why this event is good. 
%%% Prop 2 %%%
\begin{prop}\label{good}
On the event $G_m(\gamma)\cap\{(0,o) \in \cP\}$, $M_{n, \Gamma(n)} 
\ge (1+\delta) M_{m, \Gamma(m)}$ for all $n > m$. 
\end{prop}
%%%%%%%%%%
\begin{proof}
We have $M_{m+1, \Gamma(m+1)}\ge B_{m+1, \Gamma(m), \Gamma(m+1)}
M_{m, \Gamma(m)} \ge (1+\delta) M_{m, \Gamma(m)}$
by definition. Since $\Gamma$ is an open path, we have 
$M_{n, \Gamma(n)}\ge M_{m+1, \Gamma(m+1)}$ for all $n > m+1$ and the 
claim follows.  
\end{proof}
Thanks to this proposition, we have the lower bound 
\begin{equation}
	M_{n, \Gamma(n)} \ge (1+\delta)^{\sum_{m=0}^{{n-1}} 
	1_{G_m(\gamma)}}
\end{equation}
for the process along $\Gamma$ on the event $\{(0,o) \in \cP\}$. 
Therefore, the proof of Theorem~\ref{main} is reduced to proving the 
following law of large numbers 
for $\{1_{G_n(\gamma)}\}_{n \in \N}$. 

%%% Prop 3 %%%
\begin{prop}\label{LLN}
\begin{equation}
	\frac{1}{n}\sum_{m=0}^{{n-1}} 1_{G_m(\gamma)} \to c_{\delta}(B)
        \label{prop3}
\end{equation}
as $n \to \infty$ $P$-almost surely, where $c_{\delta}(B)$ is defined 
in \eqref{rate}.
\end{prop}
%%%%%%%%%%
\begin{proof}
By the independence of $B_1$ and $\{B_n\}_{n \ge 2}$ and translation 
invariance, we have
\begin{equation}
\begin{split}
	P(G_0(\gamma)) &= P(B_{1, o, x} \ge 1+\delta)P((1,x) \in \cP)\\
	&=P(B_{1, o, x} \ge 1+\delta)P((0,o) \in \cP).
\end{split}
\end{equation}
Thus~\eqref{prop3} is indeed the law of large numbers for 
$\{1_{G_n(\gamma)}\}_{n \in \N}$.

Let us first show that $\{1_{G_n(\gamma)}\}_{n \in \N}$ is a stationary 
sequence. 
Fix an increasing sequence of integers $1\le m_1<m_2< \cdots <m_k$ 
and consider the probability
$P(\bigcap_{i=1}^k G_{m_i}(\gamma))$. 
We divide the event according to the position of $\gamma({m_1})$ and use Proposition~\ref{renewal}-(a) to get 
\begin{equation}
	P\left(\bigcap_{i=1}^k G_{m_i}(\gamma)\right)=
	\sum_{y \in \Z^d}P\left(\{\gamma({m_1})=y\}\cap \bigcap_{i=1}^k G_{m_i}(\gamma^{({m_1}, y)})\right).
\end{equation}
Then by the independence of $\{\gamma({m_1})=y\} \in \cF_{1, {m_1}}$ and 
$\bigcap_{i=1}^k G_{m_i}(\gamma^{({m_1}, y)}) \in \cF_{{m_1}+1, \infty}$ and translation invariance, 
it follows that the above right hand side equals
\begin{equation}
	\sum_{y \in \Z^d}P(\gamma({m_1})=y)P\left(\bigcap_{i=1}^k 
	G_{m_i-{m_1}}(\gamma)\right)
	= P\left(\bigcap_{i=1}^k G_{m_i-{m_1}}(\gamma)\right). 
\end{equation}
Applying the same argument to $m_i'=m_i-1$ ($1\le i \le k$), 
it follows that
\begin{equation*}
P\left(\bigcap_{i=1}^k G_{m_i}(\gamma)\right)
=P\left(\bigcap_{i=1}^k G_{m_i-m_1}(\gamma)\right) 
=P\left(\bigcap_{i=1}^k G_{m_i-1}(\gamma)\right),  
\end{equation*}
which implies the stationarity of $\{1_{G_n(\gamma)}\}_{n \in \N}$. 

Next, we prove that $\{1_{G_n(\gamma)}\}_{n \in \N}$ has the so-called 
mixing property, which implies the ergodicity and hence the law of 
large numbers (cf.~\cite{PTE}, Section 6.4). 
By Lemma~6.4.4 in~\cite{PTE}, it suffices to 
show that for any pair of increasing sequences of integers 
$0\le l_1<l_2< \cdots <l_j$ and $0\le m_1<m_2< \cdots <m_k$,
\begin{equation}
	P\left(\bigcap_{i=1}^j G_{l_i}(\gamma)\cap \bigcap_{i=1}^k 
        G_{m_i+n}(\gamma)\right) \to 
	P\left(\bigcap_{i=1}^j G_{l_i}(\gamma)\right)
	P\left(\bigcap_{i=1}^k G_{m_i}(\gamma)\right)
	\label{mixing}
\end{equation}
as $n \to \infty$. 
Note first that for any $m<n$,  if 
$(m+1, \gamma(m+1)) \conn (n, \gamma(n))$ 
and $(n, \gamma(n)) \in \cP$ 
then $(m+1, \gamma(m+1)) \in \cP$. 
Combined with Proposition~\ref{renewal}-(b), this yields 
\begin{equation}
\begin{split}
	&G_{m}(\gamma)\cap G_{n}(\gamma)\\
	&\quad =\{B_{m+1, \gamma(m), \gamma(m)+x} \ge 1+\delta, 
	(m+1, \gamma(m+1)) \conn (n, \gamma(n))\} 
	\cap G_{n}(\gamma). 
\end{split}
\end{equation}
Let us denote the event in braces above by $\Tilde G_{m, n}(\gamma)$. 
Then, we can rewrite the event on the left-hand side of~\eqref{mixing} as 
\begin{equation}
\begin{split}
	\bigcap_{i=1}^j G_{l_i}(\gamma)\cap \bigcap_{i=1}^k G_{m_i+n}(\gamma)
	&=\bigcap_{i=1}^{j-1} \Tilde G_{l_i, l_{i+1}}(\gamma) 
        \cap  \Tilde G_{l_j, m_1+n}(\gamma)
        \cap \bigcap_{i=1}^{k} G_{m_i+n}(\gamma)
\end{split}
\end{equation}
provided $l_j < m_1+n$. 
Dividing the event according to the position of $\gamma(m_1+n)$, 
one can show that
\begin{equation}
\begin{split}
	& P\left(\bigcap_{i=1}^j G_{l_i}(\gamma)\cap \bigcap_{i=1}^k 
        G_{m_i+n}(\gamma)\right)\\
	&\quad=P\left(\bigcap_{i=1}^{j-1} \Tilde G_{l_i,l_{i+1}}(\gamma) 
        \cap  \Tilde G_{l_j, m_1+n}(\gamma)\right)
	P\left( \bigcap_{i=1}^{k} G_{m_i+n}(\gamma) 
        \right)\label{factorization}
\end{split}
\end{equation}
exactly in the same way as in the proof of the stationarity. Let us look at 
\begin{equation}
\begin{split}
	&\Tilde G_{l_j, m_1+n}(\gamma)\\
	&\quad =\{B_{l_j+1, \gamma(l_j), \gamma(l_j)+x} 
        \ge 1+\delta, (l_j+1, \gamma(l_j+1)) \conn (m_1+n, \gamma(m_1+n))\},
\end{split}
\end{equation}
where the dependence on $n$ remains. We recall 
Proposition~\ref{renewal}-(b) to see
\begin{equation}
	|M^{(l_j+1, \gamma(l_j+1)}_{m_1+n}| \ge 1
	\Rightarrow (l_j+1, \gamma(l_j+1)) \conn (m_1+n, \gamma(m_1+n)).
\end{equation}
Since the converse is also valid, we conclude that
\begin{equation}
	\Tilde G_{l_j, m_1+n}(\gamma)
	=\{B_{l_j+1, \gamma(l_j), \gamma(l_j)+x} \ge 1+\delta, 
        |M^{(l_j+1, \gamma(l_j+1)}_{m_1+n}| \ge 1\}
	\downarrow G_{l_j}(\gamma) 
\end{equation}
as $n \to \infty$. Coming back to~\eqref{factorization} and using 
stationarity, we obtain 
\begin{equation}
\begin{split}
	&\lim_{n \to \infty}P\left(\bigcap_{i=1}^j G_{l_i}(\gamma)\cap 
        \bigcap_{i=1}^k G_{m_i+n}(\gamma)\right)\\
	&\quad 
        =P\left( \bigcap_{i=1}^{j-1} \Tilde G_{l_i, l_{i+1}}(\gamma) 
        \cap G_{l_j}(\gamma) \right)
        P\left( \bigcap_{i=1}^k G_{m_i}(\gamma)\right)\\
	&\quad 
        =P\left( \bigcap_{i=1}^j G_{l_i}(\gamma)\right)
        P\left( \bigcap_{i=1}^k G_{m_i}(\gamma)\right)
\end{split}
\end{equation}
and we are done. 
\end{proof}
%%%%%%
\subsection{Proof of Corollary~\ref{iteration}}
\begin{proof}[Proof of Corollary~\ref{iteration}]
For the first half, we apply Theorem~\ref{main} to 
$B={\textstyle \prod_{k=1}^m} B_k$ to
find an open path $\{\Gamma(n)\}_{n \in \N}$ with respect to 
$\{{\textstyle \prod_{k=1}^m} B_{mn+k}\}_{n \in \N}$ such that 
\begin{equation}
	\liminf_{n \to \infty}\frac{1}{n} \log M_{mn, \Gamma(n)} 
	\ge c_{\delta{-\epsilon}}({\textstyle \prod_{k=1}^m} B_k)
	\log(1+\delta{-\epsilon}). 
\end{equation}
Since 
$({\textstyle \prod_{k=1}^m} B_{mn+k})_{m(n+1), \Gamma(n), \Gamma(n+1)} 
\ge 1$ implies that 
there exists a sequence $x(0)=\Gamma(n), x(1), \ldots, 
x(m-1), x(m)=\Gamma(n+1)$ such that 
\begin{equation}
	B_{mn+k, x(k-1), x(k)} \ge 1 \textrm{ for all }k=1, \ldots, m, 
\end{equation}
we can construct a path $\Gamma'$ that is open with respect to 
$\{B_n\}_{n \in \N^*}$ and
$\Gamma'(mn)=\Gamma(n)$ for all $n \in \N^*$. 
Then the claim follows from the fact $M_{mn+k, \Gamma'(mn+k)} 
\ge M_{mn, \Gamma'(mn)}$ 
for all $k=1, \ldots , m$.

We next prove the second assertion. 
If $P(|M_n| \ge 1 \textrm{ \upshape for all }n \in \N)=0$, then 
\eqref{nongrowth} is obvious. 
Suppose $\sup_{x \in \Z^d}P((\prod_{k=1}^nB)_{o,x} \ge 1+\delta)=0$
for all $\delta>0$ and $m\in\N$. 
This means that $\prod_{k=1}^nB_k$ is a binary matrix 
for all $n\in\N$ and then it follows 
$M_{n,x} \le 1$ for all $(n,x)\in \N\times\Z^d$.
Now the finite range assumption implies $\#\{x \in \Z^d:
M_{n,x} \ge 1\}\le (2r_Bn)^d$ and hence~\eqref{nongrowth}.
\end{proof}

%%% Application %%%
\section{Applications}\label{application}
\subsection{Linear stochastic evolution and its dual}
The second author has recently introduced a class of stochastic linear 
systems in~\cite{Yos08}, 
called linear stochastic evolutions (LSE for short), 
which contains various interesting 
processes such as the number of open paths in site or bond oriented 
percolation, {a} time discrete version of 
the binary contact path process, and the voter model. 
{Using our results,} we can completely characterize when the 
total mass of {an} LSE grows exponentially. 

Let us first recall the definition of {an} LSE. 
Let $A=( A_{x, y})_{x, y \in \Z^d}$ be a random matrix satisfying 
the following: 
\begin{align}
	& A_{x, y} \ge 0 \textrm{ for all } x, y \in \Z^d, \label{positivity}\\
	& \textrm{the columns } (A_{\,\cdot\,, y})_{y \in \Z^d}
        \textrm{ are independent,}\label{independence}\\
	& A_{x, y} =0 \textrm{ if } |x-y| \ge r_A 
        \textrm{ for some non-random } r_A \in \N,\label{range}\\ 
	& (A_{x+z, y+z})_{x,y \in \Z^d} \stackrel{\rm law}= A 
        \textrm{ for all } z \in \Z^d.\label{shift-inv}
\end{align} 
A few more assumptions were posed in~\cite{Yos08} such as square 
integrability of the matrix elements and 
a certain aperiodicity (cf.~(1.8) in~\cite{Yos08}) but we do not 
need them in this article. 
On the other hand, we need the following extra assumption 
\begin{equation}
	A_{x, y} \in \{0\}\cup [1, \infty) \textrm{ for all } 
        x, y \in \Z^d \label{value}
\end{equation}
to use the results in Section~\ref{Intro}. 
Let $\{A_n\}_{n \in \N^*}$ be a collection of independent copies of 
the random matrix $A$. 
The LSE generated by $\{A_n\}_{n \in \N^*}$ is {the} Markov chain 
$\{N_n\}_{n \in \N}$ with values in $[0,\infty)^{\Z^d}$ defined 
for given $N_0 \in [0,\infty)^{\Z^d}$ by 
\begin{equation}
	N_{n+1, y}=\sum_{x \in \Z^d} N_{n, x}A_{n+1, x, y}, 
        \textrm{ for }n \in \N \textrm{ and }y \in \Z^d. 
\end{equation}
If we consider the dual process of an LSE, then it can be realized 
in the same way as above 
but \eqref{independence} is replaced by
\begin{equation}
	\textrm{the rows } (A_{\,x\,, \cdot})_{x \in \Z^d}
        \textrm{ are independent,}
\end{equation}
see Section 4 in~\cite{Yos08} for detail. 
We call this type of process the dual LSE and write DLSE for short. \\

\noindent
{\it Example 1.} (Oriented percolation):
Let $(\eta_{n, y})_{(n,y) \in \N^* \times \Z^d}$ be $\{0, 1\}$-valued 
independent and identically distributed
random variables with $P(\eta_{n, y}=1)=p \in [0,1]$. 
The LSE $\{N_n\}_{n \in \N}$ generated by 
\begin{equation}
	A_{n, x,y}=\eta_{n,y} 1_{\{|x-y|=1\}}
\end{equation}
represents the number of open oriented paths up to level $n$ in the 
oriented site percolation. 
We call this process {\it the oriented site percolation} for shorthand. 
{\it The bond oriented percolation} can be constructed in a similar way 
(see p.1036 in~\cite{Yos08}). \\

\noindent
{\it Example 2.} (Binary contact path process):
Let  $(\eta_{n, y})_{(n,y) \in \N^* \times \Z^d}$ and 
$(\zeta_{n, y})_{(n,y) \in \N^* \times \Z^d}$ be $\{0, 1\}$-valued 
independent and identically distributed random variables with 
\begin{equation}
	P(\eta_{n, y}=1)=p \in [0,1] \textrm{ and }P(\zeta_{n, y}=1)=q 
        \in [0,1]. 
\end{equation}
We further introduce another family of independent and identically 
distributed random variables $(e_{n, y})_{(n,y) \in \N^* \times \Z^d}$
which are uniformly distributed on $\{e \in \Z^d: |e|=1\}$. 
Then the LSE $\{N_n\}_{n \in \N}$ generated by 
\begin{equation}
	A_{n, x,y}=\eta_{n,y} 1_{\{e_{n,y}=y-x\}}+\zeta_{n, y}\delta_{x,y} 
\end{equation}
gives a time-discrete version of the binary contact path process 
studied in~\cite{Gri83}. 
We simply call this discrete version {\it the binary contact path
process} 
in this article.
We can also define another time-discretization by considering the DLSE 
generated by 
\begin{equation}
	A_{n, x,y}=\eta_{n,x} 1_{\{e_{n,x}=y-x\}}+\zeta_{n, x}\delta_{x,y}. 
\end{equation}
In this DLSE-version, a site at time $n$ chooses the target of infection 
whereas 
in the LSE-version above, a site at time $n+1$ chooses the source of 
infection. \\

The following theorem says that 
we can characterize when LSE and DLSE grow exponentially 
in terms of $A_1$. 
%%% Theorem 2 %%%
\begin{thm}\label{LSE}
Let $\{N_n\}_{n \in \N}$ be an LSE generated by $\{A_n\}_{n \in \N^*}$ 
satisfying \eqref{value} and $|N_0|<\infty$. 
Then, either of the following holds true: 
\begin{enumerate}
\setlength{\itemsep}{10pt}
\item{If $c_{\delta}(A_1)>0$ for some  $\delta>0$, then 
\begin{equation}
	\liminf_{n \to \infty}\frac{1}{n}\log |N_n| 
        \ge c_{\delta}(A_1)\log(1+\delta) \label{exp-growth}
\end{equation}
$P$-almost surely on the event $\{|N_n| \ge 1 
\textrm{\upshape{ for all }}n \in \N\}$.}
\item{If $P(|N_n| \ge 1 \textrm{\upshape{ for all }}n \in \N)>0$ and 
$P(A_{1,-x,o}\ge 1 \textrm{\upshape{ and }}A_{1,-y,o}\ge 1)>0$ for some 
distinct $x, y \in \Z^d$, then 
\begin{equation}
	\liminf_{n \to \infty}\frac{1}{n}\log |N_n| \ge 
        \frac{1}{2}c_1(A_1A_2)\log 2>0 \label{exp-growth2}
\end{equation}
$P$-almost surely on the event $\{|N_n| \ge 1 
\textrm{\upshape{ for all }}n \in \N\}$.}
\item{If both of the assumptions in $(1)$ and $(2)$ fail, then 
\begin{equation}
	{\limsup_{n \to \infty}}\frac{1}{n}\log |N_n| \le 0 
        \label{subexp-growth}
\end{equation}
$P$-almost surely.}
\end{enumerate}
The same assertions hold for DLSE with the {second} assumption 
in {\upshape (2)} replaced by 
\begin{equation}
	P(A_{1,o, x}\ge 1 \textrm{\upshape{ and }}A_{1,o, y}\ge 1)>0 
	\textrm{\upshape{ for some distinct }}x, y \in \Z^d.\label{DLSE}
\end{equation}
\end{thm}
%%%%%%%%%%%%%
\begin{proof}
We prove (1)--(3) only for the LSE case since the proof for the 
DLSE case is almost identical. 
By linearity it suffices to show the claims for the case 
$N_0=(\delta_{o,x})_{x \in \Z^d}$. 
Note first that (1) is a direct consequence of Theorem~\ref{main}. 
Next, suppose that the assumption in (2) holds. 
Then it follows from translation invariance and \eqref{independence} that
\begin{equation}
\begin{split}
	& P((A_1A_2)_{o,x+y}\ge 2)\\
	& \quad \ge P(A_{1,o,x}\ge 1, A_{1, o, y}\ge 1, A_{2, x, x+y}
        \ge 1 \textrm{, and }A_{2, y, x+y}\ge 1)\\
	& \quad \ge P(A_{1,-x,o}\ge 1)P(A_{1, -y, o}\ge 1)P(A_{2, -x, o}
        \ge 1 \textrm{ and }A_{2, -y, o}\ge 1)\\
	& \quad >0.
\end{split}
\end{equation}
Therefore Corollary~\ref{iteration} shows \eqref{exp-growth2}.
Finally, we prove (3) by checking the assumption for~\eqref{nongrowth}. 
If both of the assumptions in $(1)$ and $(2)$ fail, 
then we have either 
\begin{equation}
	P(|N_n| \ge 1 \textrm{\upshape{ for all }}n \in \N)=0 \label{death}
\end{equation}
or
\begin{align}
  &P\left(A_{1,x,y} \in \{0,1\}\right)=1
  \textrm{ for all $x,y \in \Z^d$ and }\label{binary}\\
  &P(A_{1,-x,o}\ge 1 \textrm{ and }A_{1,-y,o}\ge 1)=0 
  \textrm{ for all distinct }x, y \in \Z^d.\label{non-cong}
\end{align}
From~\eqref{binary} and~\eqref{non-cong}
one can conclude that $(\prod_{k=1}^m A_k)_{x, y} \in \{0,1\}$ 
for all $m \in \N^*$ and $x, y \in \Z^d$. 
Thus in both cases, $c_{\delta}(\prod_{k=1}^mA_k)=0$ for all 
$\delta>0$ and $m \in \N$.
\end{proof}

Theorem~\ref{LSE} applies to both site and bond oriented percolations 
and the binary contact path process. 
Therefore we see that both the number of open paths in supercritical 
oriented percolation and 
the total mass of supercritical binary contact path process grow 
exponentially on the event of survival. 

\begin{rem}
For the oriented site/bond percolation with $d=1$, 
{it is possible to show by 
using the rightmost path instead of $\Gamma$ above that} 
\begin{equation}
	\liminf_{n \to \infty}\frac{1}{n}\log |N_n| \ge c
\end{equation}
with some absolute constant $c>0$ {almost surely on 
$\{|N_n| \ge 1 \textrm{ \upshape for all }n \in \N\}$}. 
This shows that our lower bound $c_1(A_1A_2)\log 2$ on the growth rate 
is not sharp in this case 
since it is smaller than the percolation 
probability which decreases to 0 as $p \downarrow p_c$, as was proved 
in~\cite{GH02}. 
\end{rem}

Finally, we show that surviving LSE and DLSE satisfy the assumption 
in (1) or (2) in Theorem~\ref{LSE} 
except for a few unimportant examples. 
Therefore for fairly general LSE and 
DLSE, we have the coincidence of the events 
that the process survives and that the total mass grows exponentially. 
Let us start by defining the exceptional class. 
We say that an LSE generated by $\{A_n\}_{n \in \N^*}$ is {\it trivial} 
if $A_n$ $(n \in \N^*)$ are nonrandom, that is, 
the laws of $A_n$ concentrate on one matrix. A DLSE generated by 
$\{A_n\}_{n \in \N^*}$ is said to be {\it trivial} 
if $A_n$ $(n \in \N^*)$ are nonrandom or expressed as
\begin{equation}
	A_{n, x, y}=\delta_{x+e_{n, x}, y}
\end{equation}
with $\{e_{n, x}\}_{(n,x) \in \N^*\times \Z^d}$ a family of 
$\Z^d$-valued independent and identically distributed 
random variables. In the latter case, $\{N_n\}_{n \in \N}$ is 
nothing but the coalescing random walks 
for which the questions about survival/extinction and about growth 
rate are trivial. 

\begin{cor}\label{cor-LSE}
For a nontrivial LSE or DLSE satisfying \eqref{value} and $|N_0|<\infty$, 
\begin{equation}
	\{|N_n| \ge 1 \textrm{\upshape{ for all }}n \in\N\}
	=\set{\liminf_{n \to \infty}\frac{1}{n}\log |N_n|>0}
        \label{survival=growth}
\end{equation}
modulo a $P$-null set. In particular, if 
\begin{equation}
	\lim_{n \to \infty}\frac{1}{n}E[\log(1+ |N_n|)] = 0, 
        \label{cor-LSE-nongrowth}
\end{equation}
then the process {dies out} almost surely. 
\end{cor}
\begin{proof}
Let $\{N_n\}_{n \in \N}$ be a nontrivial LSE or DLSE generated by 
$\{A_n\}_{n \in \N^*}$. 
The {second} assertion follows from \eqref{survival=growth} and 
Fatou's lemma. 
To prove \eqref{survival=growth}, it suffices to show that if the 
assumptions in (1) and (2) in Theorem~\ref{LSE} fail, then  
\begin{equation}
	P(|N_n| \ge 1 \textrm{ for all }n \in\N) = 0.\label{extinct}
\end{equation}

We begin with the LSE case. 
Suppose that both of the assumptions fail. Then, either \eqref{extinct} 
holds or $A_n$ are binary matrices 
and for any fixed $y \in \Z^d$, $A_{n,x,y}=1$ for at most one $x \in \Z^d$. 
The latter means that our process is {stochastically dominated by} 
a finite range version of the 
nearest neighbor voter model (see p.1037 in~\cite{Yos08}). 
Then, under the non-triviality assumption, one can show that it 
{dies out} almost surely by an argument 
similar to the proof of 
Lemma~1.3.3 and subsequent Remark in~\cite{Yos08}. 
Next we turn to the DLSE case. 
In this case if both of the assumptions fail, then either
 \eqref{extinct} 
holds or $A_n$ are binary matrices 
and for any fixed $x \in \Z^d$, $A_{n,x,y}=1$ for at most one $y \in \Z^d$. 
Together with the non-triviality, the latter implies that $N_n$ is a 
coalescing random walks 
killed at positive rate, which {dies out} almost surely. 
\end{proof} 
\begin{rem}
{The last part of above proof gives an alternative proof
of Theorem~\ref{LSE}-(3) for DLSE without the finite range assumption. 
Since the proofs of (1) and (2) rely only on Theorem~\ref{main} and
the first half of Corollary~\ref{iteration}, it follows that
Theorem~\ref{LSE} for DLSE holds without the finite range assumption.}

\end{rem}
%%% Continuou time case %%%
\subsection{Continuous time process}\label{CT-analogue}
We discuss the continuous time analogue of our results in this
subsection. 
More precisely, we show results similar to the preceding 
section for the continuous time version of LSE or DLSE studied 
in~\cite{NY10b,NY10a}.

Let us recall the definition of the process.
We introduce a random vector 
$K=(K_x)_{x \in \Z^d}$ such that each element takes value in 
$\{0\}\cup [1,\infty)$ and $K_x=0$ if $|x|$ is larger than
some positive constant $r_K$. 
Let us further introduce two mutually independent collections 
of i.i.d.~random variables 
$\{\tau^{z,i}\}_{z \in \Z^d, i \in \N^*}$
and $\{K^{z,i}\}_{z \in \Z^d, i \in \N^*}$
whose distributions are mean-one exponential and 
the same as that of $K$, respectively. 
We suppose that the process $\{Y_t\}_{t \ge 0}$ starts from 
$Y_0\in[0,\infty)^{\Z^d}$ and at each time 
$t=\tau^{z,1}+\tau^{z,2}+\cdots+\tau^{z,i}$ for some 
$(z,i) \in \Z^d \times \N^*$, 
the process is updated as follows:
\begin{equation}
 Y_{t,x}=\begin{cases}
	     K_{0}^{z,i}Y_{t-,z} &\textrm{if }x=z, \\
             Y_{t-,x}+K_{x-z}^{z,i}Y_{t-,z} 
             &\textrm{if }x \neq z. 
	    \end{cases}\label{update}
\end{equation}
We also consider the dual process $Z_t \in [0, \infty)^{\Z^d}$, 
$t \ge 0$ which evolves in the same way as $\{Y_t\}_{t \ge 0}$
except that \eqref{update} is replaced by its transpose:
\begin{equation}
 Z_{t,x}=\begin{cases}
	     \sum_{y \in \Z^d} K_{y-x}^{z,i}Z_{t-,y} 
             &\textrm{if }x=z, \\
             Z_{t-,x} 
             &\textrm{if }x \neq z. 
	    \end{cases}\label{dual}
\end{equation}
{
\begin{rem}
Note that the process~\eqref{update} is a continuous-time 
counterpart of DLSE, while its dual~\eqref{dual} is 
that of LSE. 
\end{rem}}
For processes of above types, we have the following simple 
characterization of exponential growth in terms of $K$.
%%% Theorem 3 %%%
\begin{thm}\label{CT}
Let $\{Y_t\}_{t \ge 0}$ be the process defined above 
satisfying $|Y_0|<\infty$.
Then, the following holds:
\begin{enumerate}
\setlength{\itemsep}{10pt}
\item{If $P(|Y_t| \ge 1 \textrm{\upshape{ for all }}t \ge 0)>0$ 
and $P(\sum_{x\in\Z^d}K_x\ge 1+\delta)>0$ for some $\delta>0$, 
then 
\begin{equation}
	\liminf_{t \to \infty}\frac{1}{t}\log |Y_t| > 0
\end{equation}
$P$-almost surely on the event 
$\{|Y_t| \ge 1 \textrm{\upshape{ for all }}t \ge 0\}$.}
\item{Otherwise, 
\begin{equation}
	{\limsup_{t \to \infty}}\frac{1}{t}\log |Y_t| \le 0 
        \label{subexp-growth2}
\end{equation}
$P$-almost surely.}
\end{enumerate}
The same assertions hold for the dual process $\{Z_t\}_{t \ge 0}$. 
\end{thm}
\begin{proof}
We prove (1) only for $\{Y_t\}_{t \ge 0}$. 
We first consider the time discretized process $\{Y_n\}_{n\in\N}$ 
and apply Theorem~\ref{main}. 
Then it can be extended to $\{Y_t\}_{t \ge 0}$ as in the proof of 
Corollary~\ref{iteration}. 

For given $\{\tau^{z,i}\}_{z \in \Z^d, i \in \N^*}$
and $\{K^{z,i}\}_{z \in \Z^d, i \in \N^*}$, 
let $B_{n+1,x,y}$ be $Y_{n+1,y}^{(n,x)}$, that is, the population
at $(n+1,y)$ starting from one particle at $(n,x)$ 
(cf.~\eqref{restart}).
Then it follows that $\{Y_n\}_{n\in\N}$ is nothing but the Markov chain 
described in Subsection~\ref{Setting and main results}. 
Therefore it suffices to check that 
$P(Y_{1,x}\ge 1+\delta)>0$ for some 
$x \in \Z^d$ when $Y_0=(\delta_{0,x})_{x \in \Z^d}$. 
Suppose first that 
$P(K_x\ge 1+\delta)>0$ for some $x \in \Z^d$. 
Then, 
\begin{equation}
\begin{split}
  &P(Y_{1,x}\ge 1+\delta)\\
  &\quad \ge P(\tau^{o,1}<1,
               \tau^{x,1}>1, 
               K^{o,1}_x \ge 1+\delta)\\
  &\quad =P(\tau^{o,1}<1)
          P(\tau^{x,1}>1) 
          P(K^{o,1}_x \ge 1+\delta)\\
  &\quad >0.
\end{split}
\end{equation}
Next, if $K_x \in\{0,1\}$ for all $x \in \Z^d$, then 
$P(\sum_{x\in\Z^d}K_x \ge 1+\delta)>0$ implies 
\begin{equation}
 P(K_x=K_y=1)>0 \textrm{ for some distinct }x,y \in \Z^d.
\end{equation}
Thus we have 
\begin{equation}
\begin{split}
  &P(Y_{1,x+y}\ge 2)\\
  &\quad \ge P(\tau^{o,1}\in[0, 1/2),
               \tau^{x,1}\in[1/2, 1),
               \tau^{y,1}\in[1/2, 1),
               \tau^{x+y,1}>1,\\
  &\quad\quad\quad\quad
   K^{o,1}_x=K^{o,1}_y=K^{x,1}_y=K^{y,1}_x=1)\\
  &\quad >0.
\end{split}
\end{equation}

The proof of (2) for $\{Y_t\}_{t \ge 0}$ is immediate since 
$P(\sum_{x\in\Z^d}K_x = 1)=1$ implies that 
$\{Y_t\}_{t \ge 0}$ is a coalescing random walk. 
To prove (2) for the dual process $\{Z_t\}_{t \ge 0}$, 
we use the fact that the cardinality of 
${\rm supp}Z_t=\{x\in\Z^d: Z_{t,x} \ge 1\}$
grows at most polynomially fast, which is proved in~\cite{Har78}. 
Indeed, ${\rm supp}Z_t$ forms an ``additive set-valued process''
introduced there and (13.10) in~\cite{Har78} implies that 
$\#{\rm supp}Z_t=O(t^d)$ almost surely. 
The rest of the proof is very similar 
to that of Theorem~\ref{LSE}-(3) and we omit the detail. 
\end{proof}
\begin{rem}
It should be pointed out that the generality of Theorem~\ref{main}
is important in this proof. 
Indeed, the matrix $B_n$ defined above is neither of finite 
range nor of independent entries. 
\end{rem}

One can also formulate and prove a continuous version of 
Corollary~\ref{cor-LSE}. 

%Ackknowledgement%%%%%%%%%%%%%%%%%%%%%%%%%%%%%%%%%
\section*{Acknowledgement} 
The authors thank professor Vladas Sidoravicius for explaining the 
main result in~\cite{KNPS10}. 
{They are also grateful to the referee for constructive comments.}

\end{document}